\newcommand{\doublespace} {\addtolength{\baselineskip}{.50\baselineskip}}

\documentclass[reqno,10pt]{amsart}
\usepackage {amsmath, amssymb, enumerate, amsthm}
\usepackage{graphicx}
\usepackage{cite}

\newtheorem{theorem}{Theorem}[section]
\newtheorem{lemma}{Lemma}[section]
\newtheorem{corollary}{Corollary}[section]

\begin{document}

\allowdisplaybreaks  % 수식이 길 경우, 통채로 다음페이지로 넘어가는 것을 방지 %

\doublespace

\title []{Finite time blow-up in higher dimensional two species problem in the Cauchy problem}

\author []{Tae Gab Ha$^\dag$ and Seyun Kim$^\ddag$}

\address{Department of Mathematics, and Institute of Pure and Applied Mathematics, Jeonbuk National University, Jeonju 54896, Republic of Korea}

\email{$^\dag$tgha@jbnu.ac.kr, $^\ddag$dccddc8097@naver.com}



%\thanks{}

\maketitle

\begin{abstract}
    In this paper, we study the blow-up radial solution of fully parabolic system with higher dimensional two species Cauchy problem for some initial condition.
    In addition, we show that the set of positive radial functions in  $L^{1}(\mathbb{R})\cap BUC(\mathbb{R}^{N}) \times L^{1}(\mathbb{R})\cap BUC(\mathbb{R}^{N}) \times W^{1,1}(\mathbb{R}^{N}) \cap W^{1,\infty}(\mathbb{R}^{N})$ has a dense subset composed of positive radial initial data causing blow-up in finite time with respect to topology  $L^{p}(\mathbb{R}^{N}) \times L^{p}(\mathbb{R}^{N}) \times H^{1}(\mathbb{R}^{N})\cap W^{1,1}(\mathbb{R}^{N})$ for $p \in \left[1,\frac{2N}{N+2}\right)$.
\end{abstract}
\section{Introduction}
Let us consider the following the two species fully parabolic system such that
\begin{equation}\label{0}
\begin{aligned}
 u_{t}&=\Delta u -\chi \nabla \cdot (u\nabla w),  \quad & x \in \Omega, t>0,\\
 v_{t}&=\Delta v - \xi \nabla \cdot (v \nabla w), \quad & x \in \Omega, t>0, \\
 w_{t}& =\Delta w - \lambda w + \alpha u +\beta v, \quad & x \in \Omega, t>0 ,
\end{aligned}
\end{equation}
where $\chi,\xi,\lambda,\alpha,\beta$ are positive constants. In this system, $(u,v)$ denote population density of the two species and $w$ denote concentration of the two species. The system (\ref{0}) is an extension of the classical Keller-Segel system because if $\chi=\xi$ and $z:=  \alpha u +\beta v$, then we have
\begin{equation}\label{00}
\begin{aligned}
 z_{t}&=\Delta z -\chi \nabla \cdot (z\nabla w),  \quad & x \in \Omega, t>0,\\
 w_{t}& =\Delta w - \lambda w + z, \quad & x \in \Omega, t>0.
\end{aligned}
\end{equation}
When $\Omega\subset \mathbb{R}^2$ is a smooth bounded domain in (\ref{00}), Nagai et al. \cite{nagai3} proved that there exists a unique radially classical solution if $\int_{\Omega}z_{0}dx < \frac{8\pi}{\chi}$ and there exists a unique classical solution if $\int_{\Omega}z_{0}dx < \frac{4\pi}{\chi}$. Horstmann proved that there exists a finite-time blow-up radial solution if $\int_{\Omega}z_{0}dx > \frac{8\pi}{\chi}$ in \cite{Hots2}, and a finite-time blow-up non-radial solution in \cite{Hots3}. In the higher dimension ($N\geqslant 3$), Winkler \cite{winkler1} showed there exists a unique classical blow-up solution for large quantity of $\int_{\Omega}z_{0}w_{0}dx$. It was also constructed the initial data set that cause a blow-up phenomenon and this set is dense in $C^{0}(\overline{\Omega}) \times W^{1,\infty}(\Omega)$ with respect to the topology $L^{p}(\Omega) \times H^{1}(\Omega)$ for $p \in \left(1,\frac{2N}{N+2}\right)$. Using the methods of \cite{winkler1}, Li et al. \cite{li1} extended the blow-up result to (\ref{0}). They constructed the new energy function in order to overcome the $\chi \neq \xi$. On the other hand, when $\Omega =\mathbb{R}^{N}$,  Winkler \cite{winkler2} showed that (\ref{00}) with $\chi=\lambda=1$ for $N\geqslant3$ blows up in finite time with $T_{max}\leqslant1$ at the origin for some assumption of the initial data. Nagai et al. \cite{nagai1} showed that (\ref{00}) has a global solution if the initial data is small enough and the solution behaves like heat kernel as $t \rightarrow \infty$ for $N\geqslant 2$. Li et al. \cite{li2} proved the similar results as \cite{nagai1} for (\ref{0}). We refer \cite{nagai5} for $N=1$. For various two species problem with parabolic-elliptic system, we refer \cite{mul1,mul2,mul3}.\\
Motivated by \cite{winkler2,winkler1,li1}, we consider the following Cauchy problem:
\begin{equation}\label{1}
\begin{aligned}
 u_{t}&=\Delta u -\chi \nabla \cdot (u\nabla w),  \quad & x \in \mathbb{R}^{N}, t>0,\\
 v_{t}&=\Delta v - \xi \nabla \cdot (v \nabla w), \quad & x \in \mathbb{R}^{N}, t>0, \\
 w_{t}& =\Delta w - \lambda w + \alpha u +\beta v, \quad & x \in \mathbb{R}^{N}, t>0 ,
\end{aligned}
\end{equation}
where $\chi,\xi,\lambda,\alpha,\beta$ are positive constants. The contents of our paper are written by using method in \cite{winkler1, winkler2}, however these cannot be applied as they are. Indeed, \cite{winkler1} constructed the energy function of system (\ref{00}) with
\begin{align*}
    \mathcal{F}(u,v)(t)=\frac{1}{2} \int_{\Omega}\vert \nabla w \vert^{2} dx +\frac{1}{2} \int_{\Omega} w^{2}dx -\int_{\Omega}zw dx +\int_{\Omega}z\log z dx.
\end{align*}
In a bounded domain, the entropy term is lower bounded such as $\int_{\Omega}z\log z dx \geqslant -\frac{\Omega}{e} $ by using $ x\log x \geqslant -\frac{1}{e}$ for all $x>0$ (see \cite{winkler1} Theorem 5.1). But, it is failed in an unbounded domain. So \cite{winkler2} considered modified entropy term by $\int_{\Omega} z \log(z+1) dx$. But on (\ref{1}) with $\chi,\xi \neq 1$, it is difficult to estimate the energy function by using  $\int_{\Omega} z \log(z+1) dx$ entropy. In order to overcome this difficulty, we construct the following energy function
\begin{equation}\label{2}
\begin{aligned}
\mathcal{F}(u,v,w)(t)&=\frac{1}{2}\int_{\mathbb{R}^{N}} \vert \nabla w(\cdot,t) \vert ^{2} dx +\frac{\lambda}{2}\int_{\mathbb{R}^{N}} w(\cdot,t)^{2}dx  -\int_{\mathbb{R}^{N}}( \alpha u(\cdot,t) +\beta v(\cdot,t) )w(\cdot,t)  dx \\
&\quad + \frac{\alpha}{\chi} \int_{\mathbb{R}^{N}} (u(\cdot,t)+1)\log(u(\cdot,t)+1) dx +\frac{\beta}{\xi}\int_{\mathbb{R}^{N}} (v(\cdot,t)+1)\log(v(\cdot,t)+1) dx,
\end{aligned}
\end{equation}
as in \cite{cal,nagai4} and \cite{li1}. Also unlike \cite{winkler2}, we will not use a cut-off function for estimating the energy functions because of additional regularity $w(\cdot,t) \in C^{1}((0,T_{max});H^{1}(\mathbb{R}^{N}))$ by using \cite{shi1,shi2,nagai2}.

Unless there is any confusion, we write as follows $\Vert \cdot \Vert_{p}:=\Vert \cdot \Vert_{L^{p}(\mathbb{R}^{N})}$ for $ 1\leqslant p \leqslant \infty$, and $BUC(\mathbb{R}^{N})$ denotes a bounded uniformly continuous Banach space on $\mathbb{R}^{N}$  with $\Vert \cdot \Vert_{\infty}$ norm. We now introduce for our main results.
\begin{theorem}\label{63}
Let $u_{0}, v_{0} \in L^{1}(\mathbb{R}^{N}) \cap BUC(\mathbb{R}^{N}), w_{0} \in W^{1,1}(\mathbb{R}^{N}) \cap W^{1,\infty}(\mathbb{R}^{N})$ be radially symmetric and nonnegative nonzero on $\mathbb{R}^{N}$ with $N\geqslant 3$. Then the classical solution of (\ref{1}) blows up at the finite time $T_{max}\leqslant1$ if $F(u_{0},v_{0},w_{0})<-CK^{\frac{2}{1-\theta}}$ , where $C$ is a for some positive constant depending on $\chi,\xi,\alpha,\beta,\lambda,N,\theta$ and $K:= \Vert u_{0} \Vert_{1}+\Vert v_{0} \Vert_{1}+\Vert w_{0} \Vert_{1}+\Vert \nabla w_{0} \Vert_{1}$ for $\theta \in \left(\frac{1}{2},1\right)$.
\end{theorem}

\begin{corollary}
Let $u_{0}, v_{0} \in L^{1}(\mathbb{R}^{N}) \cap BUC(\mathbb{R}^{N}), w_{0} \in W^{1,1}(\mathbb{R}^{N}) \cap W^{1,\infty}(\mathbb{R}^{N})$ be radially symmetric and positive on $\mathbb{R}^{N}$ with $N\geqslant 3$. Then there exists $(u_{0j})_{j \in \mathbb{N}},(v_{0j})_{j \in \mathbb{N}} \subseteq L^{1}(\mathbb{R}^{N}) \cap BUC(\mathbb{R}^{N})$ and $(w_{0j})_{j \in \mathbb{N}} \subseteq W^{1,1}(\mathbb{R}^{N}) \cap W^{1,\infty}(\mathbb{R}^{N})$ such that $(u_{0j},v_{0j},w_{0j})$ are radially symmetric and positive for all $ j \in \mathbb{N}$ with
\begin{align*}
(u_{0j},v_{0j}) &\rightarrow (u_{0},v_{0}) \ \text{in} \ L^{p}(\mathbb{R}^{N}) , \ \text{as} \ j \rightarrow \infty \ \text{for} \ p\in \left[1,\frac{2N}{N+2}\right),\\
w_{0j}& \rightarrow w_{0} \ \text{in} \ H^{1}(\mathbb{R}^{N})\cap  W^{1,1}(\mathbb{R}^{N}) ,  \ \text{as} \ j \rightarrow \infty ,
\end{align*}
and the associated classical solution ($u_{j}, v_j, w_j$) blows up at the finite time $T_{max,j}\leqslant 1$ in the sense that
\begin{equation*}
\lim sup_{t\rightarrow T_{max,j}^{+}} \Vert u_{j}(\cdot,t) \Vert_{\infty} + \Vert v_{j}(\cdot,t) \Vert_{\infty} = \infty.
\end{equation*}
\end{corollary}

\section{Proof of theorem}
In this section, we will prove main theorem. First, we deal with local well-posedness and blow-up criterion in Lemma 2.1 - 2.3. Second, we estimate the energy function and control the terms $\int \vert \nabla w \vert ^{2} dx$ and $\int (\alpha u +\beta v)w dx$ in Lemma 2.4 - 2.9. Then we will prove that the solution blows up in finite time by using contradiction arguments. Unless otherwise stated, the constant C is a generic positive constant, different in various occurrences.

\begin{lemma}\label{3}
Let $u_{0}, v_{0} \in L^{1}(\mathbb{R}^{N}) \cap BUC(\mathbb{R}^{N}), w_{0} \in W^{1,1}(\mathbb{R}^{N}) \cap W^{1,\infty}(\mathbb{R}^{N})$ be radially symmetric and nonnegative nonzero on $\mathbb{R}^{N}$ with $N\geqslant 3$. Then there exist a unique positive classical radially symmetric solution $(u,v,w)$ of (\ref{1}) such that
\begin{align*}
    u & \in C^{0}([0,T_{max});L^{1}(\mathbb{R}^{N}) \cap BUC(\mathbb{R}^{N})) \cap C^{2,1}(\mathbb{R}^{N} \times (0,T_{max})) \cap C^{0}((0,T_{max});H^{2}(\mathbb{R}^{N})), \\
    v  & \in C^{0}([0,T_{max});L^{1}(\mathbb{R}^{N}) \cap BUC(\mathbb{R}^{N})) \cap C^{2,1}(\mathbb{R}^{N} \times (0,T_{max}))\cap C^{0}((0,T_{max});H^{2}(\mathbb{R}^{N})), \\
    w & \in C^{0}([0,T_{max}); W^{1,1}(\mathbb{R}^{N}) \cap W^{1,\infty}(\mathbb{R}^{N})) \cap  C^{2,1}(\mathbb{R}^{N} \times (0,T_{max}) \cap C^{1}((0,T_{max});H^{1}(\mathbb{R}^{N})),
\end{align*}
with
\begin{align*}
    u(\cdot,t)&=e^{t\Delta}u_{0}-\chi \int_{0}^{t} \nabla \cdot e^{(t-s)\Delta}[u(\cdot,s)\nabla w(\cdot,s)]ds \quad \text{for all} \quad t\in(0,T_{max}),\\
    v(\cdot,t)&=e^{t\Delta}v_{0}-\xi \int_{0}^{t} \nabla \cdot e^{(t-s)\Delta}[u(\cdot,s)\nabla w(\cdot,s)]ds \quad \text{for all} \quad t\in(0,T_{max}),\\
    w(\cdot,t)&=e^{t(\Delta-\lambda)}w_{0}+\int_{0}^{t}e^{(t-s)(\Delta-\lambda)}[\alpha u(\cdot,s) +\beta v(\cdot,s)]ds  \quad \text{for all} \quad t\in(0,T_{max}),
\end{align*}
and if $T_{max}<\infty$, then
$$
\limsup_{t\rightarrow T_{max}} (\Vert u(\cdot,t) \Vert_{\infty}+\Vert u(\cdot,t) \Vert_{1}+\Vert v(\cdot,t) \Vert_{\infty}+\Vert v(\cdot,t) \Vert_{1}+\Vert w(\cdot,t)\Vert_{W^{1,1}(\mathbb{R}^{N})} +\Vert w(\cdot,t) \Vert_{W^{1,\infty}(\mathbb{R}^{N})}) =\infty.
$$
 Moreover,
\begin{equation}\label{4}
    \int_{\mathbb{R}^{N}} u(\cdot,t) dx =\int_{\mathbb{R}^{N}} u_{0} dx ~\text{and}~  \int_{\mathbb{R}^{N}} v(\cdot,t) dx =\int_{\mathbb{R}^{N}} v_{0} dx \quad \text{for all} \quad t \in (0,T_{max}).
\end{equation}

\end{lemma}
\begin{proof}
We obtain the existence of the mild solution by the Banach fixed point method and then this mild solution is also the classical solution by the parabolic Schauder estimate(\cite{rus} Theorem 5.1, 320p). For more detail, see \cite{winkler2,winkler3}.  Furthermore, we have $(u,v)\in C^{0}((0,T_{max});H^{2}(\mathbb{R}^{N}))$  and  $w(\cdot,t) \in C^{1}((0,T_{max});H^{1}(\mathbb{R}^{N}))$ and positivity of the solution  by virtue of method \cite{shi1}, \cite{shi2} and strong maximal principle. We can also obtain the (\ref{4}) by integral by part of the first and the second equation of (\ref{1}).

\end{proof}

\begin{lemma}\label{6}
    Let $(u_{0},v_{0},w_{0})$ satisfy the assumption in Lemma \ref{3}. Then we have
    \begin{align*}
        \Vert \nabla w \Vert_{1} &\leqslant \Vert \nabla w_{0} \Vert_{1} + C(\Vert u_{0} \Vert_{1}+\Vert v_{0} \Vert_{1}) \quad \text{for all} \quad t\in(0,T_{max}),\\
        \Vert w \Vert_{1}& \leqslant \Vert  w_{0} \Vert_{1} + C(\Vert u_{0} \Vert_{1}+\Vert v_{0} \Vert_{1}) \quad \text{for all} \quad t\in(0,T_{max}),
    \end{align*}
    where C is a for some positive constant depending on $\lambda,\alpha,\beta$.
\end{lemma}
\begin{proof}
    By using the Lemma 2.1 in \cite{winkler2} and Lemma \ref{3}, we have
    \begin{align*}
        \Vert \nabla w(\cdot,t) \Vert_{1}&= \Bigl\Vert e^{t(\Delta-\lambda)}\nabla w_{0}+\int_{0}^{t}\nabla e^{(t-s)(\Delta-\lambda)}[\alpha u(\cdot,s) +\beta v(\cdot,s)]ds \Bigr\Vert_{1} \\
          & \ \leqslant  e^{-\lambda t}\Vert \nabla w_{0} \Vert_{1}+ c_{1}\int_{0}^{t}(t-s)^{-\frac{1}{2}}e^{-\lambda(t-s)}\Vert \alpha u(\cdot,s) +\beta v(\cdot,s)\Vert_{1} ds \\
          & \ \leqslant e^{-\lambda t}\Vert \nabla w_{0} \Vert_{1} +c_{1}(\alpha \Vert u_{0} \Vert_{1}+\beta \Vert v_{0} \Vert_{1})\int_{0}^{\infty}(t-s)^{-\frac{1}{2}}e^{-\lambda(t-s)} ds .
    \end{align*}
    By the definition of the Gamma function, $\int_{0}^{\infty}(t-s)^{-\frac{1}{2}}e^{-\lambda(t-s)} ds < \infty$. Similarly, we obtain the $\Vert w \Vert_{1}$ estimate.
\end{proof}

\begin{lemma} \label{7}
Under the assumption in Lemma \ref{3}, if $T_{max} < \infty $ , then $\lim sup_{t \rightarrow T_{max}^{+}} (\Vert u \Vert_{\infty} +\Vert v \Vert_{\infty})=\infty$.
\end{lemma}
\begin{proof}
We will argue by contradiction. Suppose that there exists $C>0$ such that
\begin{align}\label{7}
        \Vert u(\cdot,t) \Vert_{\infty}, \Vert v(\cdot,t)\Vert_{\infty} \leqslant C \quad \text{for all} \quad t\in(0,T_{max}).
\end{align}
By the semigroup form solution in Lemma \ref{3}, we have
\begin{align*}
    \Vert \nabla w(\cdot,t) \Vert_{\infty}& \leqslant \Vert e^{t(\Delta -\lambda)} \nabla w_{0} \Vert_{\infty} \ +\int_{0}^{t} \Vert \nabla e^{(t-s)(\Delta-\lambda)}(\alpha u(\cdot,s)+\beta v (\cdot,s))\Vert_{\infty} ds \\
    &\leqslant\Vert \nabla w_{0} \Vert_{\infty}+(\alpha+\beta)C\int_{0}^{t} e^{-\lambda(t-s)}(t-s)^{-\frac{1}{2}}ds
\end{align*}
for all $t\in(0,T_{max})$. Thus we have
\begin{equation} \label{8}
    \Vert \nabla w(\cdot,t) \Vert_{\infty} \leqslant \Vert \nabla w_{0} \Vert_{\infty} +(\alpha+\beta)CT_{max}^{\frac{1}{2}}.
\end{equation}
Similarly, we obtain
\begin{equation} \label{9}
    \Vert w(\cdot,t)\Vert_{\infty} \leqslant \Vert w_{0} \Vert_{\infty} +(\alpha+\beta)CT_{max}.
\end{equation}
By combining (\ref{7}), (\ref{8}), (\ref{9}), Lemma \ref{6}, if $T_{max}$ is finite, then the solution of (\ref{1}) cannot blow up in finite time. This is a contradiction to blow-up criteria of Lemma \ref{3}.
\end{proof}

\begin{lemma}\label{11}
    Under the assumption as in Lemma \ref{3}, the energy functional defined in (\ref{2}) satisfies the following properties:
    \begin{align*}
        \mathcal{F}(u,v,w)(t) +\int_{0}^{t} D(s) ds &\leqslant (\alpha \chi +\beta \xi)\int_{0}^{t}\int_{\mathbb{R}^{N}} (\alpha  u(\cdot,t) +\beta  v(\cdot,s))w(\cdot,s) dx ds \\
        & \quad +(\alpha \chi+ \beta \xi)\int_{0}^{t} \mathcal{F}(u,v,w)(s) ds + \mathcal{F}(u_{0},v_{0},w_{0}),
    \end{align*}
    where
    \begin{align*}
        D(t)&:=\frac{1}{2}\int_{\mathbb{R}^{N}} w_{t}^{2}(\cdot,t) dx dt +\frac{\alpha}{2}\int_{\mathbb{R}^{N}} \left \vert \frac{\nabla u(\cdot,t)}{\sqrt{\chi(u(\cdot,t)+1)}}-\sqrt{\chi(u(\cdot,t)+1)}\right \vert^{2} dx\\
        & \quad + \frac{\beta}{2}\int_{\mathbb{R}^{N}} \left \vert \frac{\nabla v(\cdot,t)}{\sqrt{\xi(v(\cdot,t)+1)}}-\sqrt{\xi(v(\cdot,t)+1)}\right \vert^{2} dx
    \end{align*}
    for all $t \in (0,T_{max})$.
\end{lemma}
\begin{proof}
    For convenience, $u(\cdot,t),v(\cdot,t),w(\cdot,t)$ are abbreviated as $u,v,w$.
     Since $(u,v,w)$ is a classical solution of (\ref{1}) and  from that fact $(x+1)\log(x+1) \leqslant x^{2}+x$ for $x>0$ , $ \mathcal{F}(u,v,w)(t) $ is well defined and continuously differentiable on $(0,T_{max})$. So we obtain the following equality by directly differentiating.
     \begin{align*}
        \mathcal{F}'(u,v,w)(t)&=\int_{\mathbb{R}^{N}}  \nabla w \nabla w_{t} + \lambda \int_{\mathbb{R}^{N}}  w w_{t} dx\\
        & \quad -\int_{\mathbb{R}^{N}} (\alpha u_{t}+\beta v_{t})w dx -\int_{\mathbb{R}^{N}}(\alpha u+\beta v)w_{t} dx \\
        & \quad +\frac{\alpha}{\chi}\int_{\mathbb{R}^{N}} u_{t} \log(u+1) dx +\frac{\alpha}{\chi}\int_{\mathbb{R}^{N}} u_{t} dx \\
        & \quad +\frac{\beta}{\xi}\int_{\mathbb{R}^{N}} v_{t} \log(v+1) dx +\frac{\beta}{\xi}\int_{\mathbb{R}^{N}} v_{t} dx \\
        &=\int_{\mathbb{R}^{N}} \nabla w \nabla w_{t} dx +\lambda \int_{\mathbb{R}^{N}}  ww_{t} dx \\
        & \quad -\int_{\mathbb{R}^{N}}(\alpha u +\beta v) w_{t} dx -\int_{\mathbb{R}^{N}}(\alpha u_{t}+\beta v_{t})w dx \\
        & \quad +\frac{\alpha}{\chi}\int_{\mathbb{R}^{N}}u_{t} (\log(u+1)+1)dx +\frac{\beta}{\xi}\int_{\mathbb{R}^{N}}v_{t}(\log(v+1)+1)dx.
     \end{align*}
     Put
     \begin{align*}
        I_{1}&:=\int_{\mathbb{R}^{N}}  \nabla w \nabla w_{t} dx +\lambda \int_{\mathbb{R}^{N}}w w_{t} dx -\int_{\mathbb{R}^{N}}(\alpha u +\beta v) w_{t} dx,\\
        I_{2}&:=-\int_{\mathbb{R}^{N}}(\alpha u_{t}+\beta v_{t})w dx, \\
        I_{3}&:=\frac{\alpha}{\chi}\int_{\mathbb{R}^{N}}u_{t} (\log(u+1)+1)dx +\frac{\beta}{\xi}\int_{\mathbb{R}^{N}}v_{t}(\log(v+1)+1)dx.
     \end{align*}
     By integration by parts over $\mathbb{R}^{N}$, we have
     \begin{equation}\label{12}
     \begin{aligned}
        I_{1}&=-\int_{\mathbb{R}^{N}}w_{t} \Delta w dx+\lambda \int_{\mathbb{R}^{N}} w w_{t} dx
         -\int_{\mathbb{R}^{N}}(\alpha u +\beta v)w_{t}dx\\
        &=-\int_{\mathbb{R}^{N}} w_{t} \Delta w dx +\lambda \int_{\mathbb{R}^{N}}  \left(\frac{\Delta w}{\lambda}-\frac{1}{\lambda}w_{t}+\frac{\alpha}{\lambda}u+\frac{\beta}{\lambda}v \right)w_{t}dx - \int_{\mathbb{R}^{N}}(\alpha u +\beta v)w_{t} dx
        \\ &= -\int_{\mathbb{R}^{N}} w_{t}^{2} dx
     \end{aligned}
     \end{equation}
     and
     \begin{equation}\label{13}
     \begin{aligned}
        I_{2}:&=-\alpha \int_{\mathbb{R}^{N}} u_{t} w dx- \beta \int_{\mathbb{R}^{N}} v_{t}w dx \\
        & = \alpha \int_{\mathbb{R}^{N}} \nabla  w\nabla u dx -\alpha \chi \int_{\mathbb{R}^{N}} u \vert \nabla w \vert^{2} dx
         +\beta\int_{\mathbb{R}^{N}}\nabla w\nabla v dx -\beta \xi \int_{\mathbb{R}^{N}} v\vert \nabla w \vert^{2} dx.
     \end{aligned}
     \end{equation}
     By multiplying $\frac{\alpha}{\chi} (\log(u+1)+1)$ to the first equation of (\ref{1}), we have
     \begin{equation*}\label{14}
        \frac{\alpha}{\chi}\int_{\mathbb{R}^{N}}(\log(u+1)+1) u_{t} dx =-\frac{\alpha}{\chi}\int_{\mathbb{R}^{N}}\frac{\vert \nabla u \vert^{2}}{u+1} dx +\alpha \int_{\mathbb{R}^{N}}\frac{u}{u+1}\nabla u \nabla w dx.
     \end{equation*}
     Similarly, by multiplying $\frac{\beta}{\xi} (\log(v+1)+1)$ to the second equation of (\ref{1}), we have
     \begin{equation*}\label{15}
       \frac{\beta}{\xi}\int_{\mathbb{R}^{N}}(\log(v+1)+1) v_{t} dx =-\frac{\beta}{\xi}\int_{\mathbb{R}^{N}}\frac{\vert \nabla v \vert^{2}}{v+1} dx +\beta \int_{\mathbb{R}^{N}}\frac{v}{v+1}\nabla v\nabla w dx.
     \end{equation*}
     Thus
     \begin{equation}\label{17}
     \begin{aligned}
        I_{1}+I_{2}+I_{3}&=-\frac{\alpha}{2\chi} \int_{\mathbb{R}^{N}} \frac{\vert \nabla u \vert^{2}}{u+1}dx -\frac{\alpha}{2}\int_{\mathbb{R}^{N}}\left \vert \frac{\nabla u}{\sqrt{\chi(u+1)}}-\sqrt{\chi(u+1)}\nabla w \right \vert^{2} dx \\
        & \quad -\frac{\alpha \chi}{2} \int_{\mathbb{R}^{N}} u\vert \nabla w \vert^{2} dx+\frac{\alpha\chi}{2}\int_{\mathbb{R}^{N}} \vert \nabla w \vert^{2} dx+  \alpha \int_{\mathbb{R}^{N}}\frac{u}{u+1}\nabla u \nabla w dx \\
        & \quad -\frac{\beta}{2\xi} \int_{\mathbb{R}^{N}} \frac{\vert \nabla b \vert^{2}}{b+1}dx -\frac{\beta}{2}\int_{\mathbb{R}^{N}}\left \vert \frac{\nabla v}{\sqrt{\xi(v+1)}}-\sqrt{\xi(v+1)}\nabla w \right \vert^{2} dx \\
        & \quad -\frac{\beta \xi}{2} \int_{\mathbb{R}^{N}} v\vert \nabla w \vert^{2} dx+\frac{\beta\xi}{2}\int_{\mathbb{R}^{N}} \vert \nabla w \vert^{2} dx+  \beta \int_{\mathbb{R}^{N}}\frac{v}{v+1}\nabla v \nabla w dx \\
        & \quad -\int_{\mathbb{R}^{N}} w_{t}^{2} dx.
     \end{aligned}
     \end{equation}
     By Young's inequality, we obtain
    \begin{equation} \label{18}
    \begin{aligned}
        \alpha \int_{\mathbb{R}^{N}}  \frac{u}{u+1} \nabla u \nabla w dx & \leqslant \frac{\alpha}{2\chi} \int_{\mathbb{R}^{N}}\frac{\vert \nabla u \vert^{2}}{u+1}dx +
        \frac{\alpha \chi}{2}\int_{\mathbb{R}^{N}}\frac{u^{2}}{u+1}\vert \nabla w \vert^{2}dx \\
        & \leqslant  \frac{\alpha}{2\chi} \int_{\mathbb{R}^{N}}\frac{\vert \nabla u \vert^{2}}{u+1}dx  +
        \frac{\alpha \chi}{2}\int_{\mathbb{R}^{N}}u\vert \nabla w \vert^{2}dx
    \end{aligned}
    \end{equation}
    and
    \begin{equation} \label{19}
        \beta \int_{\mathbb{R}^{N}}  \frac{v}{v+1} \nabla v \nabla w dx  \leqslant  \frac{\beta}{2\xi} \int_{\mathbb{R}^{N}}\frac{\vert \nabla v \vert^{2}}{v+1}dx  +
        \frac{\beta \xi}{2}\int_{\mathbb{R}^{N}}v\vert \nabla w \vert^{2}dx.
    \end{equation}
     By inserting (\ref{18})-(\ref{19}) into (\ref{17}), we have
     \begin{equation}\label{16}
     \begin{aligned}
        \mathcal{F}'(u,v,w)(t)& \leqslant   -\frac{\alpha}{2}\int_{\mathbb{R}^{N}}\left \vert \frac{\nabla u}{\sqrt{\chi(u+1)}}-\sqrt{\chi(u+1)}\nabla w \right \vert^{2} dx \\
        & \quad  -\frac{\beta}{2}\int_{\mathbb{R}^{N}}\left \vert \frac{\nabla v}{\sqrt{\xi(v+1)}}-\sqrt{\xi(v+1)}\nabla w \right \vert^{2} dx \\
        & \quad +\frac{\alpha\chi}{2}\int_{\mathbb{R}^{N}} \vert \nabla w \vert^{2} dx\ +\frac{\beta\xi}{2}\int_{\mathbb{R}^{N}} \vert \nabla w \vert^{2} dx -\frac{1}{2}\int_{\mathbb{R}^{N}} w_{t}^{2} dx.
     \end{aligned}
     \end{equation}

     By definition of $\mathcal{F}(u,v,w)(t)$, we have
     \begin{equation}\label{20}
        \frac{\alpha \chi+\beta \xi}{2} \int_{\mathbb{R}^{N}}\vert \nabla w\vert^{2}dx \leqslant (\alpha \chi +\beta \xi)\mathcal{F}(u,v,w)(t)  +(\alpha\chi+\beta\xi)\int_{\mathbb{R}^{N}}(\alpha u +\beta v)w dx.
     \end{equation}
     Inserting (\ref{20}) into (\ref{16}) and integrating over the time , we are done.

\end{proof}

Let $\eta \in C^{\infty}(\mathbb{R})$ be nonincreasing cut-off function with $\eta=1$ in $(-\infty,0]$ and $\eta=0$ in $[1,\infty)$.
Then we can construct the radially symmetric function $\zeta_{R}$ such that
\begin{equation}\label{21}
 \zeta_{R}(x):=\eta(\vert x \vert -R) \quad \text {for} \quad R>0,
\end{equation}
with supp$(\zeta_{R}) \subseteq \overline{B_{R+1}}$ and  supp$(\nabla \zeta_{R})$ $\subseteq \overline{B_{R+1}}\backslash B_{R}$.\\

\begin{lemma}\label{70}
    The classical radially symmetry solution $w$ of (\ref{1}) satisfies the following inequality for some fixed $r_{0}(t) \in [1,2]$:\\
    If $ r \in (0,r_{0}(t)]$, then
    \begin{align}\label{22}
        w(r,t) \leqslant K+ CKr^{1-N}
    \end{align}
    and
    if $r > r_{0}(t)$, then
    \begin{align}\label{23}
        w(r,t) \leqslant K+CKr_{0}^{1-N}(t),
    \end{align}
    where $C$ is for some positive constant depending on $\lambda,\alpha,\beta$.
\end{lemma}
\begin{proof}
    By Lemma \ref{6} and mean value theorem, for $r \in (0,r_{0}(t)] $,
    \begin{equation*}
        w(r_{0},t) \leqslant r_{0}^{N-1}(t)w(r_{0}(t),t) = \int_{1}^{2}r^{N-1}w(r,t)dr \leqslant K.
    \end{equation*}
    So  we have
    \begin{align*}
        w(r,t)&=w(r_{0}(t),t)-\int_{r}^{r_{0}(t)}w_{r}(\rho,t) d\rho \\
        &\leqslant K +\int_{r}^{r_{0}(t)}\vert w_{r}(\rho,t)\vert d\rho \\
        & \leqslant K+r^{1-N}\int_{r}^{r_{0}(t)}\rho ^{N-1} \vert w_{r}(\rho,t)\vert d \rho \\
        & \leqslant K+CKr^{1-N}.
    \end{align*}
    Also if $ r > r_{0}(t)$, then
    \begin{align*}
        w(r,t) & = w(r_{0}(t),t)- \int_{r}^{r_{0}(t)}w_{r}(\rho,t)d\rho \\
        & \leqslant K+\int_{r_{0}(t)}^{r}\vert w_{r}(\rho,t)\vert d\rho \\
        & \leqslant K +r_{0}^{1-N}(t) \int_{r_{0}(t)}^{r} \rho^{N-1}\vert w_{r}(\rho,t)\vert d\rho \\
        & \leqslant K +CKr_{0}^{1-N}(t).
    \end{align*}

\end{proof}

\begin{lemma}\label{24}
    The classical radially symmetry solution of (\ref{1}) satisfies
    \begin{align*}
        \int_{\mathbb{R}^{N}}(\alpha u +\beta v) w dx \leqslant 3 \int_{B_{2}}\vert \nabla w \vert^{2} dx +CK^{2}+CK^{\frac{4}{N+4}}\Vert \Delta w -\lambda w+\alpha u +\beta v \Vert_{2}^{\frac{2N+4}{N+4}},
    \end{align*}
    where C is for some positive constant depending on $\lambda, \alpha,\beta ,N$ for $t \in (0,T_{max})$.
\end{lemma}
\begin{proof}
    Put $\zeta:=\zeta_{1}$. Due to  Lemma \ref{70}, we obtain $w(r,t) \leqslant CK $ for $r\geqslant 1 $. Hence
    \begin{equation}\label{25}
    \begin{aligned}
        \int_{\mathbb{R}^{N}}(\alpha u +\beta v)w dx & = \int_{\mathbb{R}^{N}} \zeta^{2}(\alpha u +\beta v) w dx +\int_{\mathbb{R}^{N}}(1-\zeta^{2})(\alpha u +\beta v) w dx \\
        & \leqslant \int_{\mathbb{R}^{N}} \zeta^{2}(\alpha u +\beta v) w dx + \int_{\mathbb{R}^{N}\backslash B_{1}} (\alpha u +\beta v) w dx \\
        & \leqslant \int_{\mathbb{R}^{N}} \zeta^{2}(\alpha u +\beta v) w dx  +CK \int_{\mathbb{R}^{N}}\alpha u +\beta v dx \\
        & \leqslant \int_{\mathbb{R}^{N}} \zeta^{2}(\alpha u +\beta v) w dx + CK^{2}.
    \end{aligned}
    \end{equation}
    Let $f:=-\Delta w +\lambda w -\alpha u -\beta v $. By multiplying $\zeta^{2} w $ both side of this identity, and then integral over $\mathbb{R}^{N}$ and using H\"older's and Young's inequality, we have
    \begin{align*}
        \int_{\mathbb{R}^{N}}\zeta^{2}(\alpha u +\beta v)w dx &= \int_{\mathbb{R}^{N}} \nabla w \nabla (\zeta^{2} w) dx +\int_{\mathbb{R}^{N}}\zeta^{2} w^{2} dx - \int_{\mathbb{R}^{N}}\zeta^{2} f w dx \\
        & =\int_{\mathbb{R}^{N}}\zeta^{2} \vert \nabla w \vert^{2} dx + 2\int_{\mathbb{R}^{N}}\zeta w \nabla w \nabla zeta dx +\int_{\mathbb{R}^{N}}\zeta^{2}w^{2} dx  -\int_{\mathbb{R}^{N}}\zeta^{2} f w dx \\ 
        & \leqslant 2\int_{B_{2}}\vert \nabla w \vert^{2} dx +\int_{\mathbb{R}^{N}}\vert \nabla \zeta \vert^{2} w^{2} dx +\int_{\mathbb{R}^{N}}\zeta^{2}w^{2} dx  + \Vert  f\Vert_{2} \Vert \zeta w \Vert_{2} \\
        & :=I_{1}+I_{2}+I_{3}+I_{4}.
    \end{align*}
    Since $ supp \nabla \zeta \subseteq \overline{B_{2}}\backslash B_{1}$, we get
    \begin{equation*}
        I_{2}=\int_{\mathbb{R}^{N}}\vert \nabla \zeta \vert^{2} w^{2} dx \leqslant C_{1}K^{2},
    \end{equation*}
    where $C_{1}=C^{2}\int_{\mathbb{R}^{N}}\vert \nabla \zeta \vert^{2} dx$.  \\
    Using the Gagliardo-Nirenberg and Young's inequality, we obtain
    \begin{align*}
        I_{3}=\int_{\mathbb{R}^{N}}\zeta^{2} w^{2} dx & \leqslant C\Vert \nabla(\zeta w)\Vert_{2}^{\frac{2N}{N+2}}\Vert \zeta w \Vert_{1}^{\frac{4}{N+2}} \\
        & \leqslant \frac{1}{4}\Vert \nabla(\zeta w )\Vert_{2}^{2} +C \Vert \zeta w \Vert_{1}^{2} \\
        & \leqslant \frac{1}{2} \Vert \zeta \nabla w \Vert_{2}^{2} +\frac{1}{2}\Vert w \nabla \zeta \Vert_{2}^{2}+C\Vert \zeta w \Vert_{1}^{2} \\
        & \leqslant \frac{1}{2} \Vert \zeta \nabla w \Vert_{2}^{2} +\frac{c_{3}}{2}K^{2}+CK^{2},
    \end{align*}
    \begin{align*}
        I_{4}=\Vert f \Vert_{2} \Vert \zeta w \Vert_{2}& \leqslant C\Vert f \Vert_{2} \Vert \nabla(\zeta w)\Vert_{2}^{\frac{N}{N+2}}\Vert \zeta w \Vert_{1}^{\frac{2}{N+2}} \\
        & \leqslant \frac{1}{4}\Vert \nabla(\zeta w ) \Vert_{2}^{2} +C\Vert f \Vert_{2}^{\frac{2N+4}{N+4}}\Vert \zeta w \Vert_{1}^{\frac{4}{N+4}} \\
        & \leqslant \frac{1}{2}\int_{\mathbb{R}^{N}}\zeta^{2}\vert \nabla w \vert^{2} dx +\frac{1}{2} \int_{\mathbb{R}^{N}}\vert \nabla \zeta \vert^{2} w^{2} dx +CK^{\frac{4}{N+4}}\Vert f \Vert_{2}^{\frac{2N+4}{N+4}}.
    \end{align*}
    Combining  estimates of $I_{1},I_{2},I_{3},I_{4}$ with (\ref{25}), we are done.
\end{proof}
\begin{lemma}\label{26}
    For any choice $r_{0} \in (0,2)$, the classical radially symmetry solution of (\ref{1}) satisfies
    \begin{align*}
        \int_{B_{2}\backslash B_{r_{0}}}\vert \nabla w \vert^{2}dx \leqslant \epsilon \int_{\mathbb{R}^{N}}(\alpha u +\beta v)w dx +C(\epsilon)K^{2}r_{0}^{1-N}+CKr_{0}^{\frac{1-N}{2}}\Vert w_{t} \Vert_{2}, \quad t\in (0,T_{max}),
    \end{align*}
    where $\epsilon$ is a small enough positive constant and C is a for some positive constant depending on $\lambda,\alpha,\beta,\chi,\xi,N$.
\end{lemma}
\begin{proof}
    Put $f:=-\Delta w +\lambda w -\alpha u -\beta v$ and $\zeta:=\zeta_{2}$. By multiplying both sides of $f:=-\Delta w +\lambda w -\alpha u -\beta v$ by   $\zeta^{2} w^{\frac{1}{2}}$, and then integral over $\mathbb{R}^{N}$ and using  Young's inequality, (\ref{23}) and $ supp(\nabla \zeta) \subset \overline{B_{3}}\backslash B_{2}$, we get
    \begin{equation}\label{27}
    \begin{aligned}
        \frac{1}{2}\int_{\mathbb{R}^{N}}\zeta^{2} w^{-\frac{1}{2}}\vert \nabla w \vert^{2} &=-2\int_{\mathbb{R}^{N}}\zeta w^{-\frac{1}{2}} \nabla w \nabla \zeta +\int_{\mathbb{R}^{N}}\zeta^{2}(\alpha u +\beta v)w^{\frac{1}{2}}dx  \\
        & \quad -\lambda \int_{\mathbb{R}^{N}}\zeta^{2} w^{\frac{3}{2}} dx +\int_{\mathbb{R}^{N}}\zeta^{2} f w^{\frac{1}{2}} \\
        & \leqslant \frac{1}{4}\int_{\mathbb{R}^{N}}\zeta^{2} w^{-\frac{1}{2}}\vert \nabla w \vert^{2} dx +4 \int_{\mathbb{R}^{N}}\vert \nabla \zeta \vert^{2} w^{\frac{3}{2}}dx \\
        & \quad + \int_{\mathbb{R}^{N}}\zeta^{2}(\alpha u +\beta v) w^{\frac{1}{2}} dx +\int_{\mathbb{R}^{N}}\zeta^{2} \vert f \vert w^{\frac{1}{2}} dx \\
        & \leqslant \frac{1}{4}\int_{\mathbb{R}^{N}}\zeta^{2} w^{-\frac{1}{2}}\vert \nabla w \vert^{2} dx +4CK^{\frac{3}{2}} \\
        & \quad + \int_{\mathbb{R}^{N}}\zeta^{2}(\alpha u +\beta v) w^{\frac{1}{2}} dx +\int_{\mathbb{R}^{N}}\zeta^{2} \vert f \vert w^{\frac{1}{2}} dx. \\
    \end{aligned}
    \end{equation}
    On the other hand, let $x \in B_{2}\backslash B_{r_{0}}$ with $0<r_{0}<2$.  If $0<r_{0}<1$ , then by using the (\ref{22}), we have
    \begin{align*}
        v(x,t)& \leqslant K+CK \vert x \vert^{1-N} \\
        &\leqslant Kr_{0}^{1-N} +CKr_{0}^{1-N} \\
        & \leqslant CKr_{0}^{1-N}.
    \end{align*}
    If $1 \leqslant r_{0}<2$, then using the (\ref{23}), we get
    \begin{align*}
        w(x,t)& \leqslant K+CKr_{0}^{1-N} \\
        & \leqslant 2^{N-1}Kr_{0}^{1-N} +CKr_{0}^{1-N} \\
        & \leqslant CKr_{0}^{1-N}.
    \end{align*}
    Thus we have $w(x,t) \leqslant CKr_{0}^{1-N}$  for $x\in B_{2}\backslash B_{r_{0}}$. By using this fact and the fact that $\zeta^{2}=1$ on $B_{2}$, we have
    \begin{equation}\label{28}
        \frac{1}{4}\int_{\mathbb{R}^{N}}\zeta^{2}w^{-\frac{1}{2}}\vert \nabla w \vert^{2} dx \geqslant \frac{1}{4}C^{-\frac{1}{2}}K^{-\frac{1}{2}}r_{0}^{\frac{N-1}{2}}\int_{B_{2}\backslash B_{r_{0}}} \vert \nabla w \vert^{2} dx. \\
    \end{equation}
    Inserting (\ref{28}) into (\ref{27}), we have
    \begin{equation}\label{29}
        \int_{B_{2}\backslash B_{r_{0}}} \vert \nabla w \vert^{2} dx \leqslant CK^{2}r_{0}^{\frac{1-N}{2}} +CK^{\frac{1}{2}}r_{0}^{\frac{1-N}{2}}\int_{\mathbb{R}^{N}}(\alpha u +\beta v) w^{\frac{1}{2}} dx + CK^{\frac{1}{2}}r_{0}^{\frac{1-N}{2}} \int_{\mathbb{R}^{N}}\vert f \vert w^{\frac{1}{2}} dx.
    \end{equation}
    By Young's inequality, we have
    \begin{equation}\label{30}
    \begin{aligned}
        CK^{\frac{1}{2}}r_{0}^{\frac{1-N}{2}}\int_{\mathbb{R}^{N}}(\alpha u +\beta v) w^{\frac{1}{2}} dx & \leqslant \epsilon \int_{\mathbb{R}^{N}}(\alpha u +\beta v)w dx  +C(\epsilon) Kr_{0}^{1-N}\int_{\mathbb{R}^{N}}(\alpha u +\beta v) dx \\
        & \leqslant \epsilon \int_{\mathbb{R}^{N}}(\alpha u +\beta v)w dx +C(\epsilon) K^{2}r_{0}^{1-N}.
    \end{aligned}
    \end{equation}
    By H\"older's inequality, we get
    \begin{align}\label{31}
        CK^{\frac{1}{2}}r_{0}^{\frac{1-N}{2}} \int_{\mathbb{R}^{N}}\vert f \vert w^{\frac{1}{2}} dx & \leqslant CK^{\frac{1}{2}}r_{0}^{\frac{1-N}{2}}\Vert f \Vert_{2} \Vert w^{\frac{1}{2}} \Vert_{2} \\ \nonumber
        &= CK^{\frac{1}{2}}r_{0}^{\frac{1-N}{2}}\Vert f \Vert_{2} \Vert w \Vert_{1}^{\frac{1}{2}} \\ \nonumber
        & \leqslant CKr_{0}^{\frac{1-N}{2}}\Vert f \Vert_{2}.
    \end{align}
    By inserting (\ref{30}) and (\ref{31}) into (\ref{29}) and using the fact that $2^{\frac{1-N}{2}} \leqslant r_{0} ^{\frac{1-N}{2}}$, we complete the proof.
\end{proof}

\begin{lemma}\label{32}
    For  $r_{0} \in (0,2)$ and $t \in (0,T_{max})$, the classical radially symmetry solution of (\ref{1}) satisfies
    \begin{align*}
        \int_{B_{r_{0}}} \vert \nabla w \vert^{2} dx & \leqslant C(N,\alpha,\beta,\chi,\xi,N)K + C(N)r_{0}\Vert \Delta w
         -\lambda w +\alpha u +\beta v \Vert_{2}^{2} \\
         & \quad + C(\alpha,\chi,N)\sqrt{K} \left \Vert \frac{\nabla u}{\sqrt{\chi(u+1)}}-\sqrt{\chi(u+1)}\nabla w \right \Vert_{2} \\
         & \quad + C(\beta,\xi,N)\sqrt{K} \left \Vert \frac{\nabla v}{\sqrt{\xi(v+1)}}-\sqrt{\xi(v+1)}\nabla w \right \Vert_{2} \\
         & \quad +2\lambda \int_{B_{2}} w^{2} dx.
    \end{align*}

\end{lemma}
\begin{proof}
    Let
    \begin{align*}
        f & \equiv f(r,t):=-\Delta w +\lambda w -\alpha u -\beta v, \\
        g_{1} & \equiv g_{1}(r,t):= \frac{u_{r}}{\sqrt{\chi(u+1)}}-\sqrt{\chi(u+1)}w_{r}, \\
        g_{2} & \equiv g_{2}(r,t):=\frac{v_{r}}{\sqrt{\xi(v+1)}}-\sqrt{\xi(v+1)}w_{r}.
    \end{align*}
    Multiplying both sides of the following equation  by $2r^{2N-2}w_{r}$,
    \begin{equation*}
        r^{1-N}(r^{N-1}w_{r})_{r} = -\alpha u -\beta v +\lambda w -f,
    \end{equation*}
    we have
    \begin{equation*}
        \partial_{r}(r^{2N-2} w_{r}^{2})=-2\alpha r^{2N-2}uw_{r}-2\beta r^{2N-2}vw_{r}+2\lambda r^{2N-2}w w_{r} -2r^{2N-2}fw_{r}.
    \end{equation*}
    By using the fact that
    \begin{align*}
        \chi u w_{r} &= u_{r}-\chi w_{r} -\sqrt{\chi(u+1)}g_{1}, \\
        \xi  v w_{r} &= v_{r}-\xi w_{r}-\sqrt{\xi(v+1)}g_{2},
    \end{align*}
    and  Young's inequality, we have
    \begin{equation}\label{33}
    \begin{aligned}
        \partial_{r}(r^{2N-2} w_{r}^{2})&=-\frac{2\alpha}{\chi}r^{2N-2}u_{r}+2\alpha r^{2N-2}w_{r}+\frac{2\alpha}{\chi}r^{2N-2}\sqrt{\chi(u+1)}g_{1} \\
        & \quad -\frac{2\beta}{\xi}r^{2N-2}v_{r}+2\beta r^{2N-2}w_{r}+\frac{2\beta}{\xi}r^{2N-2}\sqrt{\xi(v+1)}g_{2} \\
        & \quad +2\lambda r^{2N-2}ww_{r}-2r^{2N-2}fw_{r} \\
        &  \leqslant - \frac{2\alpha}{\chi}r^{2N-2}u_{r}+2\alpha r^{2N-2}w_{r}+\frac{2\alpha}{\chi}r^{2N-2}\sqrt{\chi(u+1)}g_{1} \\
        & \quad -\frac{2\beta}{\xi}r^{2N-2}v_{r}+2\beta r^{2N-2}w_{r}+\frac{2\beta}{\xi}r^{2N-2}\sqrt{\xi(v+1)}g_{2} \\
        & \quad +2\lambda r^{2N-2}ww_{r}+(N-1)r^{2N-2}w_{r}^{2}+\frac{1}{N-1}r^{2N-2}f^{2}.
    \end{aligned}
    \end{equation}
    Thus by the Gronwall inequality, we get
    \begin{equation}\label{34}
    \begin{aligned}
        r^{2N-2}w_{r}^{2}(r,t) & \leqslant \int_{0}^{r}e^{(N-1)(r-\rho)} \big[- \frac{2\alpha}{\chi}\rho^{2N-2}u_{r}(\rho,t)+2\alpha \rho^{2N-2}w_{r}(\rho,t) \\
        & \quad -\frac{2\beta}{\xi}\rho^{2N-2}v_{r}(\rho,t)+2\beta \rho^{2N-2}w_{r}(\rho,t) \\
        & \quad +\frac{2\alpha}{\chi}\rho^{2N-2}\sqrt{\chi(u(\rho,t)+1)}g_{1}(\rho,t) +\frac{2\beta}{\xi}\rho^{2N-2}\sqrt{\xi(v(\rho,t)+1)}g_{2}(\rho,t) \\
        & \quad +2\lambda \rho^{2N-2}w(\rho,t)w_{r}(\rho,t)+\frac{1}{N-1}\rho^{2N-2}f^{2}(\rho,t)\big]  d\rho.
    \end{aligned}
    \end{equation}
    By integration by parts, we have
    \begin{equation}\label{35}
    \begin{aligned}
        &\int_{0}^{r}e^{(N-1)(r-\rho)} \big[- \frac{2\alpha}{\chi}\rho^{2N-2}u_{r}(\rho,t)+2\alpha \rho^{2N-2}w_{r}(\rho,t) \\
        & \quad -\frac{2\beta}{\xi}\rho^{2N-2}v_{r}(\rho,t)+2\beta \rho^{2N-2}w_{r}(\rho,t)
          +2\lambda \rho^{2N-2}w(\rho,t)w_{r}(\rho,t)\big] d\rho \\
         &= \big[e^{(N-1)(r-\rho)}(-\frac{2\alpha}{\chi}\rho^{2N-2}u(\rho,t)+2\alpha \rho^{2N-2}w(\rho,t) -\frac{2\beta}{\xi}\rho^{2N-2}v(\rho,t)\\
         & \quad +2\beta \rho^{2N-2}w(\rho,t)+2\lambda \rho^{2N-2}w^{2}(\rho,t)) \big]_{0}^{r}+(N-1)\int_{0}^{r}e^{(N-1)(r-\rho)}\rho^{2N-3}(2-\rho) \\
         & \quad  \cdot \left[\frac{2\alpha}{\chi}u(\rho,t)-2\alpha w(\rho,t)+\frac{2\beta}{\xi}v(\rho,t)-2\beta w(\rho,t)-2\lambda \rho^{2N-2}w^{2}(\rho,t) \right]d\rho \\
         & \leqslant 2\alpha r^{2N-2}w(r,t)+2\beta r^{2N-2}w(r,t)+2\lambda r^{2N-2}w^{2}(r,t) \\
         & \quad +4(N-1)\int_{0}^{r}e^{(N-1)(r-\rho)}\rho^{2N-3}\left(\frac{2\alpha}{\chi}u(\rho,t)+\frac{2\beta}{\xi}v(\rho,t)\right) d\rho \\
         & \leqslant 2\alpha r^{2N-2}w(r,t)+2\beta r^{2N-2}w(r,t)+2\lambda r^{2N-2}w^{2}(r,t) \\
         & \quad + 4(N-1)e^{2(N-1)}\int_{0}^{r}\rho^{2N-3} \left[ \frac{2\alpha}{\chi}u(\rho,t)+\frac{2\beta}{\xi}v(\rho,t)\right] d\rho,
    \end{aligned}
    \end{equation}
    for $r \in (0,2)$.
    In addition, using H\"older inequality, we obtain that
    \begin{equation}\label{36}
    \begin{aligned}
        \int_{0}^{r}& e^{(N-1)(r-\rho)}\left( \frac{2\alpha}{\chi}\rho^{2N-2}\sqrt{\chi(u(\rho,t)+1)}g_{1}(\rho,t)\right) d\rho  \\
        & \leqslant \frac{2\alpha}{\chi}r^{N-1}e^{2(N-1)}\int_{0}^{r} \rho^{\frac{N-1}{2}}\sqrt{\chi(u(\rho,t)+1)}g_{1}(\rho,t)\rho^{\frac{N-1}{2}} d\rho \\
        & \leqslant \frac{2\alpha}{\chi}r^{N-1}e^{2(N-1)} \left[ \int_{0}^{r} \chi \rho^{N-1}(u(\rho,t)+1) d\rho \right]^{\frac{1}{2}}\cdot \left[ \int_{0}^{r} \rho^{N-1}g_{1}^{2}(\rho,t) d\rho \right]^{\frac{1}{2}} \\
        & \leqslant  \frac{2\alpha}{\chi}r^{N-1}\sqrt{\chi}\left(\Vert u \Vert_{1}+\frac{2^{N}}{N}\right)^{\frac{1}{2}}\cdot \left(\int_{0}^{\infty}\rho^{N-1}g_{1}^{2}(\rho,t) d\rho \right)^{\frac{1}{2}} \\
        & \leqslant  \frac{2\alpha}{\chi}r^{N-1}C(N)\sqrt{\chi}K^{\frac{1}{2}}\left(\int_{0}^{\infty}\rho^{N-1}g_{1}^{2}(\rho,t) d\rho \right)^{\frac{1}{2}},
    \end{aligned}
    \end{equation}
    for $r \in (0,2)$.
    By the same argument as (\ref{36}), we have
    \begin{equation}\label{37}
         \int_{0}^{r} e^{(N-1)(r-\rho)}\left( \frac{2\beta}{\xi}\rho^{2N-2}\sqrt{\xi(v(\rho,t)+1)}g_{2}(\rho,t)\right) d\rho  \leqslant  \frac{2\beta}{\xi}r^{N-1}\sqrt{\xi}C(N)K^{\frac{1}{2}}\left(\int_{0}^{\infty}\rho^{N-1}g_{2}^{2}(\rho,t) d\rho \right)^{\frac{1}{2}}
    \end{equation}
    and
    \begin{equation}\label{38}
        \frac{1}{N-1}\int_{0}^{r}e^{(N-1)(r-\rho)}\rho^{2N-2}f^{2}(\rho,t) d\rho \leqslant \frac{1}{N-1}e^{2(N-1)}r^{N-1}\int_{0}^{r}\rho^{N-1}f^{2}(\rho,t)d\rho,
    \end{equation}
     for $r \in (0,2)$. Inserting (\ref{35}) - (\ref{38}) to (\ref{34}), we get
     \begin{equation}\label{39}
     \begin{aligned}
         r^{2N-2}w_{r}^{2}(r,t) & \leqslant 2\alpha r^{2N-2}w(r,t)+2\beta r^{2N-2}w(r,t)+2\lambda r^{2N-2}w^{2}(r,t) \\
         & \quad + 4(N-1)e^{2(N-1)}\int_{0}^{r}\rho^{2N-3} \left[ \frac{2\alpha}{\chi}u(\rho,t)+\frac{2\beta}{\xi}v(\rho,t)\right] d\rho  \\
         & \quad +  \frac{2\alpha}{\sqrt{\chi}}r^{N-1}C(N)K^{\frac{1}{2}}\left(\int_{0}^{\infty}\rho^{N-1}g_{1}^{2}(\rho,t) d\rho \right)^{\frac{1}{2}} \\
         & \quad +  \frac{2\beta}{\sqrt{\xi}}r^{N-1}C(N)K^{\frac{1}{2}}\left(\int_{0}^{\infty}\rho^{N-1}g_{2}^{2}(\rho,t) d\rho \right)^{\frac{1}{2}} \\
         & \quad + \frac{1}{N-1}e^{2(N-1)}r^{N-1}\int_{0}^{\infty}\rho^{N-1}f^{2}(\rho,t)d\rho.
     \end{aligned}
     \end{equation}
     Multiplying (\ref{39}) by $r^{1-N}$ and integrating from $0$ to $r_0$ with $r_{0} \in (0,2)$, we have
     \begin{equation}\label{40}
     \begin{aligned}
         &\int_{0}^{r_{0}}\rho^{N-1}w_{r}^{2}(\rho,t) d\rho \\
         & \leqslant 2\alpha \int_{0}^{r_{0}}\rho^{N-1}w(\rho,t)d\rho+2\beta \int_{0}^{r_{0}}\rho^{N-1}w(\rho,t)d\rho+2\lambda \int_{0}^{r_{0}}\rho^{N-1}w^{2}(\rho,t)d\rho \\
         & \quad + 4(N-1)e^{2(N-1)}\int_{0}^{r_{0}}r^{1-N}\int_{0}^{r}\rho^{2N-3} \left[ \frac{2\alpha}{\chi}u(\rho,t)+\frac{2\beta}{\xi}v(\rho,t)\right] d\rho dr  \\
         & \quad +  \frac{2\alpha}{\sqrt{\chi}}C(N)K^{\frac{1}{2}}r_{0}\left(\int_{0}^{\infty}\rho^{N-1}g_{1}^{2}(\rho,t) d\rho \right)^{\frac{1}{2}} \\
         & \quad +  \frac{2\beta}{\sqrt{\xi}}C(N)K^{\frac{1}{2}}r_{0}\left(\int_{0}^{\infty}\rho^{N-1}g_{2}^{2}(\rho,t) d\rho \right)^{\frac{1}{2}} \\
         & \quad + \frac{1}{N-1}e^{2(N-1)}r_{0}\int_{0}^{\infty}\rho^{N-1}f^{2}(\rho,t)d\rho.
     \end{aligned}
     \end{equation}
     By Fubini's theorem, we obtain that
     \begin{equation}\label{41}
     \begin{aligned}
        4&(N-1)e^{2(N-1)}\int_{0}^{r_{0}}r^{1-N}\int_{0}^{r}\rho^{2N-3} \left[ \frac{2\alpha}{\chi}u(\rho,t)+\frac{2\beta}{\xi}v(\rho,t)\right] d\rho dr  \\
        & = 4(N-1)e^{2(N-1)}\int_{0}^{r_{0}}\int_{\rho}^{r_{0}}r^{1-N}\rho^{2N-3} \left[ \frac{2\alpha}{\chi}u(\rho,t)+\frac{2\beta}{\xi}v(\rho,t)\right] dr d\rho  \\
        & = \frac{4(N-1)e^{2(N-1)}}{N-2}\int_{0}^{r_{0}}(\rho^{2-N}-r_{0}^{2-N})\rho^{2N-3} \left[ \frac{2\alpha}{\chi}u(\rho,t)+\frac{2\beta}{\xi}v(\rho,t)\right] d\rho \\
        & \leqslant \frac{4(N-1)e^{2(N-1)}}{N-2} \int_{0}^{r_{0}} \rho^{N-1} \left[ \frac{2\alpha}{\chi}u(\rho,t)+\frac{2\beta}{\xi}v(\rho,t)\right] d\rho \\
        & \leqslant C(\alpha,\beta,\chi,\xi,N)K.
     \end{aligned}
     \end{equation}
     Thus the proof is completed.
\end{proof}

\begin{lemma}\label{42}
    For $\theta  \in \left(\frac{1}{2},1\right)$, the classical radially symmetry solution of (\ref{1}) satisfies
    \begin{equation*}
        \frac{1}{2}\int_{\mathbb{R}^{N}}(\alpha u +\beta v)w dx \leqslant CK^{2}(\Vert f \Vert_{2}^{2\theta}+\Vert g_{1} \Vert_{2} +\Vert g_{2} \Vert_{2}+1), \ \text{for} \ t \in (0,T_{max}),
    \end{equation*}
    where C is a for some constant depending on $\lambda,\alpha,\beta,\chi,\xi,N$ and
    \begin{align*}
        f &=-\Delta w +\lambda w -\alpha u -\beta v, \\
        g_{1} & = \frac{\nabla u}{\sqrt{\chi(u+1)}}-\sqrt{\chi(u+1)}\nabla w, \\
        g_{2} & =\frac{\nabla v}{\sqrt{\xi(v+1)}}-\sqrt{\xi(v+1)}\nabla w.
    \end{align*}
\end{lemma}
\begin{proof}
Let $\alpha \in \left(0,\frac{2}{N-1}\right)$ and $r_{0} \equiv r_{0}(t):= \min \{1,\Vert f \Vert_{2}^{-\alpha}\}$. By Lemma \ref{26} and \ref{32} with taking $\epsilon=\frac{1}{12}$, we have
\begin{equation}\label{43}
\begin{aligned}
    \int_{B_{2}} \vert \nabla w \vert^{2} dx & \leqslant \frac{1}{12}\int_{\mathbb{R}^{N}}(\alpha u+ \beta v)w dx+CK^{2}r_{0}^{1-N} + CKr_{0}^{\frac{1-N}{2}}\Vert f \Vert_{2} \\
    & \quad + CK + Cr_{0}\Vert f \Vert_{2}^{2} + C\sqrt{K} \Vert g_{1} \Vert_{2} + C\sqrt{K}\Vert g_{2} \Vert_{2}+\int_{B_{2}} w^{2}dx,
\end{aligned}
\end{equation}
where $C$ is a for some positive constant. Using the property that $W^{1,2}(B_{2})$ is compactly imbedding to $L^{2}(B_{2})$ and $L^{2}(B_{2})$ is continuously imbedding to $L^{1}(B_{2})$ to apply Ehrling-type Lemma, we can get
\begin{equation}\label{44}
    \int_{B_{2}} w^{2} dx \leqslant \frac{1}{2} \int_{B_{2}}\vert \nabla w \vert^{2} dx + C\left( \int_{B_{2}} w dx \right)^{2}\leqslant \frac{1}{2}\int_{B_{2}}\vert \nabla w \vert^{2} dx +CK^{2},
\end{equation}
where $C$ is a for some positive constant.
Inserting (\ref{44}) into the (\ref{43}), we have
\begin{align}\label{45}
    \frac{1}{2}\int_{B_{2}} \vert \nabla w \vert^{2} dx & \leqslant \frac{1}{12}\int_{\mathbb{R}^{N}}(\alpha u+ \beta v)w dx+CK^{2}r_{0}^{1-N} + CKr_{0}^{\frac{1-N}{2}}\Vert f \Vert_{2} \\ \nonumber
    & \quad + CK + Cr_{0}\Vert f \Vert_{2}^{2} + C\sqrt{K} \Vert g_{1} \Vert_{2} + C\sqrt{K}\Vert g_{2} \Vert_{2}+CK^{2}.
\end{align}
By Lemma \ref{24}, we have
\begin{equation}\label{46}
    \int_{\mathbb{R}^{N}}(\alpha u +\beta v) w dx \leqslant 3 \int_{B_{2}}\vert \nabla w \vert^{2} dx +CK^{2}+CK^{\frac{4}{N+4}}\Vert f \Vert_{2}^{\frac{2N+4}{N+4}}.
\end{equation}
Inserting (\ref{45}) into (\ref{46}), we get
\begin{equation}\label{47}
\begin{aligned}
    \frac{1}{2}\int_{\mathbb{R}^{N}}(\alpha u +\beta v)w dx & \leqslant 6CK^{2}r_{0}^{1-N}+6CK+6CK^{2}+CK^{2}\\
    & \quad +6CKr_{0}^{\frac{1-N}{2}}\Vert f \Vert_{2}+6Cr_{0}\Vert f \Vert_{2}^{2}+CK^{\frac{4}{N+4}}\Vert f \Vert_{2}^{\frac{2N+4}{N+4}} \\
    & \quad + 6C\sqrt{K}\Vert g_{1} \Vert_{2} +6C\sqrt{K}\Vert g_{2} \Vert_{2}.
\end{aligned}
\end{equation}
If $\Vert f \Vert_{2}>1$ for $t \in (0,T_{max})$, then $r_{0}=\Vert f \Vert_{2}^{-\alpha}$. So we obtain
\begin{equation}\label{48}
\begin{aligned}
    \frac{1}{2}\int_{\mathbb{R}^{N}}(\alpha u +\beta v)w dx & \leqslant 6CK^{2}\Vert f \Vert_{2}^{\alpha(N-1)}+6CK+6CK^{2}+CK^{2}\\
    & \quad +6CK\Vert f \Vert_{2}^{1+\frac{\alpha(N-1)}{2}}+6C\Vert f \Vert_{2}^{2-\alpha}+CK^{\frac{4}{N+4}}\Vert f \Vert_{2}^{\frac{2N+4}{N+4}} \\
    & \quad + 6C\sqrt{K}\Vert g_{1} \Vert_{2} +6C\sqrt{K}\Vert g_{2} \Vert_{2}.
\end{aligned}
\end{equation}
Set $\theta:=max\{\frac{\alpha(N-1)}{2},\frac{1}{2}+\frac{\alpha(N-1)}{4},\frac{2-\alpha}{2},\frac{N+2}{N+4}\}$.
Since $\alpha \in \left(0,\frac{2}{N-1}\right)$, $\theta \in \left(\frac{1}{2},1\right)$. Thus by applying Young's inequality to the term $\Vert f \Vert_{2}$ of (\ref{48}) and using the fact $K>1$, we have
\begin{equation}\label{49}
     \frac{1}{2}\int_{\mathbb{R}^{N}}(\alpha u +\beta v)w dx  \leqslant CK^{2}(\Vert f \Vert_{2}^{2\theta}+1+\Vert g_{1}\Vert_{2}+\Vert g_{2} \Vert_{2}),
\end{equation}
where $C$ is a for some positive constant.\\
If $\Vert f \Vert_{2} \leqslant 1$, then $r_{0}(t)=1$.
Thus by (\ref{47}), we get
\begin{align*}
    \frac{1}{2}\int_{\mathbb{R}^{N}}(\alpha u +\beta v)w dx & \leqslant 6CK^{2}+6CK+6CK^{2}+CK^{2}\\
    & \quad +6CK\Vert f \Vert_{2}+6C\Vert f \Vert_{2}^{2}+CK^{\frac{4}{N+4}}\Vert f \Vert_{2}^{\frac{2N+4}{N+4}} \\
    & \quad + 6C\sqrt{K}\Vert g_{1} \Vert_{2} +6C\sqrt{K}\Vert g_{2} \Vert_{2}.
\end{align*}
By the same argument as (\ref{49}), we get
\begin{equation*}
     \frac{1}{2}\int_{\mathbb{R}^{N}}(\alpha u +\beta v)w dx \leqslant CK^{2}(\Vert f \Vert_{2}^{2\theta}+1+\Vert g_{1}\Vert_{2}+\Vert g_{2} \Vert_{2}).
\end{equation*}
This completes the proof.

\end{proof}

Next, we will introduce the Gronwall type inequality.

\begin{lemma}[\cite{winkler2}] \label{57}
Suppose that $a>0 ,b>0$, and $\kappa>1$ and that for some $T>0$, an nonnegative function $y \in C^{0}([0,T])$ satisfies
\begin{equation*}
    y(t) \geqslant a + b \int_{0}^{t} y^{\kappa}(s) ds \quad \text{for all} \quad t\in (0,T).
\end{equation*}
Then
\begin{equation*}
    T \leqslant \frac{1}{(\kappa-1)a^{\kappa-1}b}.
\end{equation*}

\end{lemma}

We now prove Theorem 1.1 by using the previous Lemmas.

\textit{\textbf{Proof of Theorem 1.1}}

We will briefly write the energy function such as $\mathcal{F}(u,v,w)(t):=\mathcal{F}(t), \mathcal{F}(u_{0},v_{0},w_{0}):=\mathcal{F}(0)$. Due to $ \theta \in \left(\frac{1}{2},1\right)$, Lemma \ref{42} and Young's inequality, we have  for $t \in (0,T_{max})$,
 \begin{equation}\label{50}
    \int_{\mathbb{R}^{N}} (\alpha u + \beta v)w dx \leqslant C_{1}K^{2}(D^{\theta}(t)+1) \leqslant \frac{1}{2(\alpha \chi+\beta \xi)}D(t)+ C_{2}K^{\frac{2}{1-\theta}},
 \end{equation}
 where $C_{1},C_{2}$ are for some positive constant depending on $\lambda,\alpha,\beta,\chi,\xi,N,\theta$. \\
 Assume that $\mathcal{F}(0) \leqslant -C_{3}K^{\frac{2}{1-\theta}}$ where $C_{3}$ satisfies
 \begin{equation}\label{51}
    C_{3} > \max\Bigl\{2C_{1}+(\alpha \chi+\beta \xi)C_{2}, (\alpha \chi+\beta \xi)C_{2}+ \left(\frac{2\theta}{1-\theta}2^{\frac{\theta+1}{\theta}}C_{1}^{\frac{1}{\theta}}\right)^{\frac{\theta}{1-\theta}}\Bigr\}.
 \end{equation}
In order to prove that $T_{max} \leqslant 1$, we will argue by contradiction. If it was false, that is, $T_{max}>1$ and let
 \begin{equation*}\label{52}
    t_{0}:= \sup\{ \tilde{t} \in (0,1) \vert \mathcal{F}(t) < -2C_{1}K^{2} \quad \text{for} \quad \forall t \in (0,\tilde{t}) \}.
 \end{equation*}
 $t_0$ is well defined due to $\frac{2}{1-\theta} >2, K>1$ and continuity of $\mathcal{F}(t)$ so we deduce that
 \begin{equation*}
    \mathcal{F}(0) <-(2C_{1}+(\alpha \chi+\beta \xi)C_{2})K^{\frac{2}{1-\theta}} < -2C_{1}K^{\frac{2}{1-\theta}}  <-2C_{1}K^{2}.
 \end{equation*}
By definition of $t_{0}$, $t_{0}=1$. Indeed, if $t_{0}<1$, then we have
\begin{equation}\label{53}
    \mathcal{F}(t) < -2C_{1}K^{2} \quad \text{for all} \quad t \in(0,t_{0}), \quad \text{and} \quad \mathcal{F}(t_{0})=-2C_{1}K^{2}.
\end{equation}
Since $T_{max}>1$ , Lemma \ref{11} and (\ref{50}), we can get
\begin{equation}\label{54}
\begin{aligned}
    \mathcal{F}(t) & \leqslant -\int_{0}^{t}D(s)ds +\mathcal{F}(0) + (\alpha \chi +\beta \xi) \int_{0}^{t}\mathcal{F}(s)ds \\
    & \quad + (\alpha \chi + \beta \xi) \int_{0}^{t}\int_{\mathbb{R}^{N}}(\alpha u +\beta v) w dx ds \\
    & \leqslant -\frac{1}{2}\int_{0}^{t}D(s)ds+\mathcal{F}(0)+(\alpha \chi +\beta \xi)C_{2}K^{\frac{2}{1-\theta}},
\end{aligned}
\end{equation}
for $t \in (0,t_{0}]$.
Thus by (\ref{53}) and(\ref{51}), we obtain
\begin{equation*}
    -2C_{1}K^{2}=\mathcal{F}(t_{0})  \leqslant \mathcal{F}(0)+(\alpha \chi+\beta \xi) C_{2}K^{\frac{2}{1-\theta}}  < (-C_{3}+(\alpha \chi +\beta \xi)C_{2})K^{\frac{2}{1-\theta}} < -2C_{1}K^{\frac{2}{1-\theta}} <-2C_{1}K^{2}.
\end{equation*}
This is a contradiction. Thus $t_{0}=1$. \\
On the other hand, from (\ref{50}), we can get
\begin{equation*}
    \mathcal{F}(t) \geqslant -\int_{\mathbb{R}^{N}}(\alpha u +\beta v)w dx \geqslant -C_{1}K^{2}(D^{\theta}(t)+1)
\end{equation*}
for $t \in (0,t_{0})$. Then we have by (\ref{53}),
\begin{equation}\label{55}
    D(t) \geqslant \left(\frac{-\mathcal{F}(t)}{C_{1}K^{2}}-1\right)^{\frac{1}{\theta}} \geqslant \left(\frac{-\mathcal{F}(t)}{2C_{1}K^{2}}\right)^{\frac{1}{\theta}}
\end{equation}
for $t\in(0,t_{0})$.\\
According to (\ref{54}),(\ref{55}) and (\ref{51}), $\mathcal{F}(t)$ satisfies
\begin{equation}\label{56}
\begin{aligned}
    -\mathcal{F}(t) & \geqslant \frac{1}{2}\int_{0}^{t}\left(\frac{-\mathcal{F}(s)}{2C_{1}K^{2}}\right)^{\frac{1}{\theta}}ds-\mathcal{F}(0)-(\alpha \chi+\beta \xi)C_{2}K^{\frac{2}{1-\theta}} \\
    &= 2^{-\frac{\theta+1}{\theta}}C_{1}^{-\frac{1}{\theta}}K^{-\frac{2}{\theta}}\int_{0}^{t}(-\mathcal{F}(s))^{\frac{1}{\theta}}ds-\mathcal{F}(0)-(\alpha \chi+\beta \xi)C_{2}K^{\frac{2}{1-\theta}} \\
    & \geqslant 2^{-\frac{\theta+1}{\theta}}C_{1}^{-\frac{1}{\theta}}K^{-\frac{2}{\theta}}\int_{0}^{t}(-\mathcal{F}(s))^{\frac{1}{\theta}}ds +(C_{3}-(\alpha \chi+\beta \xi)C_{2})K^{\frac{2}{1-\theta}} \\
    & \geqslant 2^{-\frac{\theta+1}{\theta}}C_{1}^{-\frac{1}{\theta}}K^{-\frac{2}{\theta}}\int_{0}^{t}(-\mathcal{F}(s))^{\frac{1}{\theta}}ds +\left(\frac{2\theta}{1-\theta}2^{\frac{\theta+1}{\theta}}C_{1}^{\frac{1}{\theta}}\right)^{\frac{\theta}{1-\theta}}K^{\frac{2}{1-\theta}}
\end{aligned}
\end{equation}
for $t \in (0,t_{0})$.
Combining (\ref{56}) with Lemma \ref{57}, $t_{0}$ must satisfy the following inequality
\begin{align*}
    t_{0} \leqslant \frac{1}{\left(\frac{1-\theta}{\theta}\right)\left(\left(\frac{2\theta}{1-\theta}2^{\frac{\theta+1}{\theta}}C_{1}^{\frac{1}{\theta}}\right)^{\frac{\theta}{1-\theta}}K^{\frac{2}{1-\theta}}\right)^{\frac{1-\theta}{\theta}}2^{-\frac{\theta+1}{\theta}}C_{1}^{-\frac{1}{\theta}}K^{-\frac{2}{\theta}}}=\frac{1}{2}.
\end{align*}
It is contradiction to $t_{0}=1$, consequently, $T_{max}>1$ is false. This completes the proof.

\begin{flushright}
  $\Box$
\end{flushright}

Finally, we will prove Corollary 1.1.

\textit{\textbf{Proof of Corollary 1.1}}

    Fix $(r_{j})_{j\in\mathbb{N}} \subseteq (0,1)$ such that $r_{j} \rightarrow 0$ \ as \ $j \rightarrow 0$ and put \\
    \begin{equation}\label{58}
        \varphi(\epsilon):=\int_{0}^{1}\rho^{N-1}(\rho^{2}+\epsilon)^{-\frac{N}{2}} d\rho \quad \text{for} \quad \epsilon >0 .
    \end{equation}
    Then we have $\varphi(\epsilon) \rightarrow \infty $  \ as \ $\epsilon \rightarrow 0$, by monotone converge theorem. So we can fix $\eta_{j} \in (0,1)$ small enough such that
    \begin{equation}\label{59}
        r_{j}^{N}\varphi(\frac{\eta_{j}}{r_{j}^{2}}) \geqslant j \quad \text{for all} \quad j\in\mathbb{N}.
    \end{equation}
    Next, pick any $\kappa \in \left(N-\frac{N}{p},\frac{N-2}{2}\right)$ and let
    $$
    u_{0j}=
    \begin{cases}
        a_{j}(r^{2}+\eta_{j})^{-\frac{N-\kappa}{2}}, & \mbox{if } r\in [0,r_{j}],\\
        u_{0}(r), & \mbox{if } r>r_{j},
    \end{cases}
    $$
    $$
    v_{0j}=
    \begin{cases}
        b_{j}(r^{2}+\eta_{j})^{-\frac{N-\kappa}{2}}, & \mbox{if } r\in [0,r_{j}],\\
        v_{0}(r), & \mbox{if } r>r_{j},
    \end{cases}
    $$
    $$
    w_{0j}=
    \begin{cases}
        c_{j}(r^{2}+\eta_{j})^{-\frac{\kappa}{2}}, & \mbox{if } r\in [0,r_{j}],\\
        w_{0}(r), & \mbox{if } r>r_{j},
    \end{cases}
    $$
    where $a_{j}:=(r_{j}^{2}+\eta_{j})^{\frac{N-\kappa}{2}}u_{0}(r_{j}) , \  b_{j}:=(r_{j}^{2}+\eta_{j})^{\frac{N-\kappa}{2}}v_{0}(r_{j}) $ and $ c_{j}:=(r_{j}^{2}+\eta_{j})^{\frac{\kappa}{2}}w_{0}(r_{j}).$ Then for $p \in \left[1,\frac{2N}{N+2}\right]$ and by using the fact $\vert r_{j}^{2}+\eta_{j} \vert \leqslant 2$, we obtain
    \begin{equation}\label{60}
    \begin{aligned}
        \Vert u_{0j} \Vert_{L^{p}(B_{r_{j}})}^{p}&=\vert \partial B_{1} \vert  \int_{0}^{r_{j}}r^{N-1}a_{j}^{p}(r^{2}+\eta_{j})^{-\frac{p(N-\kappa)}{2}}dr \\
        & = \vert \partial B_{1} \vert u_{0j}^{p}(r_{j})\int_{0}^{r_{j}}r^{N-1}\left(\frac{r_{j}^{2}+\eta_{j}}{r^{2}+\eta_{j}}\right)^{\frac{p(N-\kappa)}{2}}dr \\
        & \leqslant \vert \partial B_{1} \vert u_{0j}^{p}(r_{j})2^{\frac{p(N-\kappa)}{2}}\int_{0}^{r_{j}}r^{N-1-p(N-\kappa)} dr \\
        & = \frac{\leqslant \vert \partial B_{1} \vert u_{0j}^{p}(r_{j})2^{\frac{p(N-\kappa)}{2}}}{N-p(N-\kappa)}r_{j}^{N-p(N-\kappa)}.
    \end{aligned}
    \end{equation}
    Since $ \kappa >N-\frac{N}{p}$, we have $\Vert u_{0j} \Vert_{L^{p}(B_{r_{j}})} \rightarrow 0 \quad \text{as} \quad j \rightarrow \infty$. Hence
    \begin{equation*}
        \Vert u_{0,j}-u_{0} \Vert_{p} = \Vert u_{0,j}-u_{0} \Vert_{L^{p}(B_{r_{j}})} \leqslant \Vert u_{0j} \Vert_{L^{p}(B_{r_{j}})} +\Vert u_{0}\Vert_{\infty}\vert B_{r_{j}} \vert^{\frac{1}{p}}
    \end{equation*}
    Thus $u_{0j} \rightarrow u_{0} $ in $L^{p}(\mathbb{R}^{N})$ as $j \rightarrow \infty$. Similarly, $ v_{0j} \rightarrow v_{0} $ in $L^{p}(\mathbb{R}^{N})$ as $j \rightarrow \infty$. \\
    For $r \in [0,r_{j}]$, we can get $w'_{0j}(r)=-\kappa c_{j}r(r^{2}+\eta_{j})^{-\frac{\kappa+2}{2}}$. Hence
    \begin{align*}
        \Vert \nabla w_{0j} \Vert_{L^{2}(B_{r_{j}})}^{2} &= \vert \partial B_{1} \vert \kappa^{2} c_{j}^{2} \int_{0}^{r_{j}}r^{N+1}(r^{2}+\eta_{j})^{-(\kappa+2)}dr \\
        & \leqslant \vert \partial B_{1} \vert \kappa^{2} 2^{\kappa}\Vert w_{0}\Vert_{\infty}^{2}\int_{0}^{r_{j}}r^{N+1-2(\kappa+2)}dr \\
        &= \frac{\vert \partial B_{1} \vert \kappa^{2}2^{\kappa}\Vert w_{0}\Vert_{\infty}^{2}}{N+2-2(\kappa+2)}r_{j}^{N+2-2(\kappa+2)}.
    \end{align*}
    Since $\kappa< \frac{N-2}{2}$, we get $N+2-2(\kappa+2)>0$. Thus $\Vert \nabla w_{0j} \Vert_{L^{2}(B_{r_{j}})} \rightarrow 0$ as $ j \rightarrow \infty$.
    Similarly, $\Vert w_{0j} \Vert_{L^{2}(B_{r_{j}})} \rightarrow 0$ as $  j \rightarrow \infty$. Therefore,
    \begin{equation*}
        \Vert w_{0j}-w_{0} \Vert_{H^{1}(\mathbb{R}^{N})}=\Vert w_{0j}-w_{0} \Vert_{H^{1}(B_{r_{j}})}\leqslant \Vert w_{0j} \Vert_{H^{1}(B_{r_{j}})}+C\Vert w_{0} \Vert_{W^{1,\infty}(\mathbb{R}^{N})}\vert B_{r_{j}}\vert^{\frac{1}{2}},
    \end{equation*}
    where $C$ is a for some constant depending on $N$. Thus it conclude that
    \begin{equation}\label{61}
        w_{0j} \rightarrow w_{0} \quad \text{in} \quad H^{1}(\mathbb{R}^{N}) \quad \text{as} \quad j \rightarrow \infty.
    \end{equation}
    Since $L^{2}(B_{r_{j}})$ is continuous imbedding to $L^{1}(B_{r_{j}})$, $\Vert w_{0j} \Vert_{W^{1,1}(B_{r_{j}})} \rightarrow 0$ as $ j \rightarrow \infty$ and
    \begin{align*}
        \Vert w_{0j}-w_{0} \Vert_{W^{1,1}(\mathbb{R}^{N})}&=\Vert w_{0j}-w_{0} \Vert_{W^{1,1}(B_{r_{j}})}\\
        &\leqslant \Vert w_{0j} \Vert_{W^{1,1}(B_{r_{j}})}+C\Vert w_{0} \Vert_{W^{1,\infty}(\mathbb{R}^{N})}\vert B_{r_{j}}\vert,
    \end{align*}
     where $C$ is a for some constant depending on $N$. Thus we obtain
    \begin{equation*}
        w_{0j} \rightarrow w_{0} \quad \text{in} \quad W^{1,1}(\mathbb{R}^{N}) \quad \text{as} \quad j \rightarrow \infty.
    \end{equation*}
    We note that  L'Hospital's rule shows $ \sup_{\epsilon >0}\frac{\epsilon \log(\epsilon+1)}{\epsilon^{p}} < \infty$ for $p \in \left(1,\frac{2N}{N+2}\right)$ and $\sup_{\epsilon >0}\frac{\log(\epsilon+1)}{\epsilon} <\infty$.
    Since $(u_{0j},v_{0j})$ converge in $ L^{p}(\mathbb{R}^{N})$ for $p \in \left[1,\frac{2N}{N+2}\right)$ ,
    \begin{equation*}
        \int_{\mathbb{R}^{N}}(u_{0j}+1)\log(u_{0j}+1)dx ~\text{and}~ \int_{\mathbb{R}^{N}}(v_{0j}+1)\log(v_{0j}+1)dx < \infty \quad  \forall j \in \mathbb{N}.
    \end{equation*}
    Also by (\ref{61}), we obtain
    \begin{equation*}
        \int_{\mathbb{R}^{N}} \vert \nabla w_{0j} \vert^{2}dx ,\int_{\mathbb{R}^{N}} w_{0j}^{2} dx < \infty \quad \forall j \in \mathbb{N}.
    \end{equation*}
     Therefore,
    \begin{multline}\label{62}
        \sup_{j\in \mathbb{N}} \Bigl\{ \frac{1}{2} \int_{\mathbb{R}^{N}} \vert \nabla w_{0j} \vert^{2} dx + \frac{\lambda}{2}\int_{\mathbb{R}^{N}} w_{0j}^{2} dx
        \\
        +\frac{\alpha}{\chi}\int_{\mathbb{R}^{N}}(u_{0j}+1)\log(u_{0j}+1)dx+\frac{\beta}{\xi}\int_{\mathbb{R}^{N}}(v_{0j}+1)\log(v_{0j}+1)dx \Bigr\} < \infty.
    \end{multline}
    On the other hand, by (\ref{59}), we obtain
    \begin{align*}\label{63}
        \int_{\mathbb{R}^{N}}(\alpha u_{0j}+\beta v_{0j})w_{0j}dx & \geqslant \int_{B_{r_{j}}}(\alpha u_{0j}+\beta v_{0j})w_{0j}dx \\
        &= \vert \partial B_{1} \vert(\alpha a_{j}c_{j}+\beta b_{j}c_{j})\int_{0}^{r_{j}}r^{N-1}(r^{2}+\eta_{j})^{-\frac{N}{2}}dr \\
        & = \vert \partial B_{1} \vert(\alpha a_{j}c_{j}+\beta b_{j}c_{j})\int_{0}^{1}r_{j}^{N}\rho^{N-1}(r_{j}^{2}\rho^{2}+\eta_{j})^{-\frac{N}{2}}d\rho \\
        & =  \vert \partial B_{1} \vert(\alpha a_{j}c_{j}+\beta b_{j}c_{j})\int_{0}^{1}\rho^{N-1}\left(\rho^{2}+\frac{\eta_{j}}{r_{j}^{2}}\right)^{-\frac{N}{2}}d\rho \\
        & = \vert \partial B_{1} \vert(\alpha a_{j}c_{j}+\beta b_{j}c_{j})\varphi(\frac{\eta_{j}}{r_{j}^{2}}) \\
        &=\vert \partial B_{1} \vert (\alpha u_{0}(r_{j})+\beta v_{0}(r_{j}))(r_{j}^{2}+\eta_{j})^{\frac{N}{2}}w_{0j}(r_{j})\varphi(\frac{\eta_{j}}{r_{j}^{2}}) \\
        & \geqslant j\vert \partial B_{1} \vert (\alpha u_{0}(r_{j})+\beta v_{0}(r_{j}))w_{0j}(r_{j}),
    \end{align*}
    which implies that
    \begin{equation}\label{64}
      \int_{\mathbb{R}^{N}}(\alpha u_{0j}+\beta v_{0j})w_{0j}dx \rightarrow \infty \quad \text{as} \quad j \rightarrow \infty.
    \end{equation}
    Combining (\ref{64}) with (\ref{62}), we have
    \begin{equation*}
        \mathcal{F}(u_{0j},v_{0j},w_{0j}) \rightarrow -\infty \quad \text{as} \quad j\rightarrow \infty.
    \end{equation*}
    Put
    \begin{equation*}
        K_{j}:=\Vert u_{0j} \Vert_{1}+\Vert v_{0j} \Vert_{1}+\Vert w_{0j} \Vert_{1} + \Vert \nabla w_{0j} \Vert_{1}+1.
    \end{equation*}
     Since  $(u_{0j},v_{0j}) \rightarrow (u_{0},v_{0})$  in $L^{p}(\mathbb{R}^{N})$ for $p \in \left[1,\frac{2N}{N+2}\right)$ as $j \rightarrow \infty$,
    and $w_{0j} \rightarrow w_{0}$  in $H^{1}(\mathbb{R}^{N}) \cap W^{1,1}(\mathbb{R}^{N})$ as $j \rightarrow \infty$,  we have
    \begin{equation*}
        K_{j} \rightarrow K:=\Vert u_{0} \Vert_{1}+\Vert v_{0} \Vert_{1}+\Vert w_{0} \Vert_{1} + \Vert \nabla w_{0} \Vert_{1}+1, \quad \text{as} \quad j \rightarrow \infty.
    \end{equation*}
    Thus $K_{j}$ is bounded, consequently, there exists positive integer $M$  such that

    \begin{equation*}
        \mathcal{F}(u_{0j},v_{0j},w_{0j}) \leqslant -CK_{j}^{\frac{2}{1-\theta}}  \quad \forall j>M,
    \end{equation*}
     where $C$ is for some positive constant depending on $\lambda,\alpha,\beta,\chi,\xi,N,\theta \in \left(\frac{1}{2},1\right)$ .
    If we replace $j$ with $j-M$, the classical radially symmetry solutions ($u_{j},v_{j},w_{j}$) corresponding to ($u_{0j},v_{0j},w_{0j}$) blows up at the finite time with $T_{max,j}\leqslant 1$ for $\forall j \in \mathbb{N}$.

\begin{flushright}
  $\Box$
\end{flushright}

\end{document}